\documentclass[a4paper,12pt]{article}
\usepackage[top=2.5cm,bottom=2.5cm,left=2.5cm,right=2.5cm]{geometry}
\usepackage{cite, amsmath, amssymb}
\usepackage{amssymb}
\usepackage{graphicx}
\newenvironment{proof}{{\noindent \it Proof.}}{\hfill $\blacksquare$\par}
\usepackage{latexsym}
\usepackage{graphicx,booktabs,multirow}
\usepackage{tikz}
\usetikzlibrary{decorations.pathreplacing}
\usetikzlibrary{intersections}
\usepackage{enumerate}
\usepackage{graphicx,booktabs,multirow}
\usepackage{appendix}
\newtheorem{theorem}{Theorem}[section]

\newtheorem{lemma}[theorem]{Lemma}
\newtheorem{corollary}[theorem]{\rm\bfseries Corollary}

\newtheorem{problem}{Problem}[section]
\newtheorem{fact}[theorem]{Fact}
\usepackage[section]{placeins}
\allowdisplaybreaks

\begin{document}

\vspace*{10mm}

\noindent
{\Large \bf Sharp bounds on the symmetric division deg index of graphs and line graphs}

\vspace{7mm}

\noindent
{\large \bf Hechao Liu$^{1,*}$, Yufei Huang$^2$}

\vspace{7mm}

\noindent
$^1$ School of Mathematical Sciences, South China Normal University, Guangzhou, 510631, P. R. China,
e-mail: {\tt hechaoliu@m.scnu.edu.cn} \\[2mm]
$^2$ Department of Mathematics Teaching, Guangzhou Civil Aviation College, Guangzhou, 510403, P. R. China, e-mail: {\tt fayger@qq.com}
\\[2mm]
$^*$ Corresponding author

\vspace{7mm}

\noindent
{\bf Abstract} \
For a graph $G$ with vertex set $V_{G}$ and edge set $E_{G}$, the symmetric division deg index is defined as $SDD(G)=\sum\limits_{uv\in E_{G}}(\frac{d_{u}}{d_{v}}+\frac{d_{v}}{d_{u}})$, where $d_{u}$ denotes the degree of vertex $u$ in $G$.
In 2018, Furtula et al. confirmed the quality of SDD index exceeds that of some more popular VDB indices, in particular that of the GA index. They shown a close connection between the SDD index and the earlier well-established GA index. Thus it is meaningful and important to consider the chemical and mathematical properties of the SDD index.
In this paper, we determine some sharp bounds on the symmetric division deg index of graphs and line graphs and characterize the corresponding extremal graphs.
\vspace{5mm}



\baselineskip=0.30in

\section{Introduction}

We use $|U|$ to denote the cardinality of set $U$.
Let $G=(V_{G},E_{G})$ be a graph with vertex set $V_{G}$ and edge set $E_{G}$.
Let $n_{G}:=|V_{G}|$ and $m_{G}:=|E_{G}|$ be the order and size of $G$, respectively.
Denote by $N_{G}(u)$ the neighbors of vertex $u$, $d_{G}(u):=|N_{G}(u)|$ the degree of vertex $u$ in $G$.
We use $\Delta_{G}$ and $\delta_{G}$ to denote the maximum degree and minimum degree in $G$, respectively.
We call $G$ is a $\delta$-regular graph if $d_{u}=\delta$ for any $u\in V_{G}$.
A $(\Delta,\delta)$-biregular graph is the bipartite graph with $d_{u}=\Delta$, $d_{v}=\delta$ for any $uv\in E_{G}$.
For convenience, we sometimes write $d_{G}(u)$ as $d_{u}$ without causing confusion.
If $E_{G}\neq \emptyset$, we call $G$ is a nontrivial graph, we only consider connected nontrivial graphs in this paper.
Denote by $C_{n}$, $K_{n}$, $S_{n}$ and $P_{n}$, the cycle, complete graph, star graph and path with order $n$, respectively.
In this paper, all notations and terminologies used but not defined can refer to Bondy and Murty \cite{bond2008}.

The line graph $\mathcal{L}(G)$ is the graph whose vertices set is the edge sets of $G$ and two vertices in $\mathcal{L}(G)$ are adjacent if the corresponding two edges has one common vertex in $G$.
We use $\Delta_{G}$ (resp., $\delta_{G}$) to denote the maximum degree (resp., minimum degree) of graph $G$. We use $\Delta_{\mathcal{L}(G)}$ (resp., $\delta_{\mathcal{L}(G)}$) to denote the maximum degree (resp., minimum degree) of line graph $\mathcal{L}(G)$.

The first and second Zagreb indices \cite{gutr1972} are defined as
\begin{eqnarray*}
M_{1}(G)=\sum\limits_{uv\in E_{G}}(d_{u}+d_{v})=\sum\limits_{u\in V_{G}}d_{u}^{2},\ \ M_{2}(G)=\sum\limits_{uv\in E_{G}}d_{u}d_{v}.
\end{eqnarray*}
They are often used to study molecular complexity, chirality, and other chemical properties.
Others see \cite{dehy2007,guda2004,gfvp2015,helz2019} and the references within.

The first and second general Zagreb indices \cite{brsg2014,lizh2005} are defined as
\begin{eqnarray*}
M_{1}^{\alpha}(G)=\sum\limits_{u\in V_{G}}d_{u}^{\alpha},\ \ M_{2}^{\alpha}(G)=\sum\limits_{uv\in E_{G}}(d_{u}d_{v})^{\alpha},
\end{eqnarray*}
with $\alpha\in R$.

The general sum-connectivity index \cite{zhtr2010} is defined as
\begin{eqnarray*}
\chi_{\alpha}(G)=\sum\limits_{uv\in E_{G}}(d_{u}+d_{v})^{\alpha},
\end{eqnarray*}
with $\alpha\in R$.

The geometric-arithmetic index (GA) \cite{vufr2009} is defined as
\begin{eqnarray*}
GA(G)=\sum\limits_{uv\in E_{G}}\frac{2\sqrt{d_{u}d_{v}}}{d_{u}+d_{v}}.
\end{eqnarray*}

In 2010, Vuki\v{c}evi\'{c} and Ga\v{s}perov proposed the symmetric division deg index (SDD) \cite{vuga2010}, which is defined as
\begin{eqnarray*}
SDD(G)=\sum\limits_{uv\in E_{G}}(\frac{d_{u}}{d_{v}}+\frac{d_{v}}{d_{u}}).
\end{eqnarray*}
Since then, the SDD index has attracted much attention of researchers.
Furtula et al. \cite{fudg2018} showed that the SDD index gains the comparable correlation coefficient with a well-known geometric-arithmetic index, the applicative potential of SDD is comparable to already well-established VDB structure descriptors.
Vasilyev \cite{vasi2014} determined lower and upper bounds of symmetric division deg index in some classes of graphs and determine the corresponding extremal graphs.
Das et al. \cite{dama2019} obtained some new bounds for SDD index and presented a relation between SDD index and other topological indices.
Pan et al. \cite{pali2019} determined the extremal SDD index among trees, unicyclic graphs and bicyclic graphs.
They also determined the minimum SDD index of tricyclic graphs \cite{lpli2020}.
Ali et al. \cite{alem2020} characterized the graphs with fifth to ninth minimum SDD indices from the class of all molecular trees.
One can refer to \cite{dusu2021,gulo2016,glsr2016,pjic2019,gzam2021,rase2020,sgdu2021} for more details about SDD index.

The relations between GA index (resp. AG index, general sum-connectivity index, Harmonic index) and the line graphs had been considered in \cite{cgpp2020,cpst2020,ligz2022,pest2019}.
We take further the line by investigating the SDD index.
In this paper, we first determine some sharp bounds on the SDD index of graphs, then determine some sharp bounds on the SDD index of line graphs.
In this paper, we only consider connected nontrivial graphs.
Let $\mathcal{G}_{n}$ be the set of connected nontrivial graphs with order $n$, $\mathcal{G}_{n,m}$ the set of connected nontrivial graphs with order $n$ and size $m$.

\section{Preliminaries}

Recall that we only consider connected nontrivial graphs in this paper.
We write graphs to denote connected nontrivial graphs without causing confusion.

\begin{lemma}\label{l2-2}
Let $f(x,y)=\frac{x}{y}+\frac{y}{x}$, and real number $a,b$ satisfied that $0<a\leq x\leq y\leq b$.
Then $1\leq f(x,y)\leq \frac{a}{b}+\frac{b}{a}$, with left equality if and only if $x=y$, right equality if and only if $x=a, y=b$.
\end{lemma}
\begin{proof}
The binary functions $f(x,y)=\frac{x}{y}+\frac{y}{x}$ and $0<a\leq x\leq y\leq b$.
Let $g(t)=t+\frac{1}{t}$ $(t\geq 1)$. Since $g'(t)=1-\frac{1}{t^{2}}\geq 0$, then $g(t)$ is monotonically increasing for $t\geq 1$.

Since $0<a\leq x\leq y\leq b$, then $1\leq \frac{y}{x}\leq \frac{b}{a}$.
Thus $2=g(1)\leq f(x,y)=g(t)\leq g(\frac{b}{a})=\frac{a}{b}+\frac{b}{a}$,
with left equality if and only if $x=y$, right equality if and only if $x=a, y=b$.
\end{proof}

\begin{lemma}\label{l2-3}\rm(\cite{ross2018}\rm)
Let $G$ be a graph with maximum degree $\Delta$ and minimum
degree $\delta$, and $\alpha>0$. Then
$$\frac{\delta^{\alpha}}{2}M_{1}^{\alpha+1}(G)\leq M_{2}^{\alpha}(G)\leq  \frac{\Delta^{\alpha}}{2}M_{1}^{\alpha+1}(G),  $$
with both equalities hold if and only if $G$ is regular.
\end{lemma}

\begin{lemma}\label{l2-4}\rm(\cite{pest2019}\rm)
Let $G$ be a graph and $G\ncong P_{n}$. Then $m_{G}\leq m_{\mathcal{L}(G)}$.
\end{lemma}

\begin{lemma}\label{l2-5}\rm(\cite{cgpp2020}\rm)
Let $G$ be a graph. Then $m_{\mathcal{L}(G)}= \frac{1}{2}M_{1}(G)-m_{G}$.
\end{lemma}

We also need these simple Facts in the proof of our results.
\begin{fact}\label{f2-6}\rm(\cite{ligz2022}\rm)

$(i)$ If $\mathcal{L}(G)\cong S_{2}$, then $G\cong P_{3}$;

$(ii)$ If $\mathcal{L}(G)\cong S_{3}$, then $G\cong P_{4}$;

$(iii)$ If $\mathcal{L}(G)\cong S_{n}$ $(n\geq 4)$, then $G=\emptyset$;

$(iv)$ If $\mathcal{L}(G)\cong C_{3}$, then $G\cong C_{3}$ or $S_{4}$;

$(v)$ If $\mathcal{L}(G)\cong C_{n}$ $(n\geq 4)$, then $G=\emptyset$;

$(vi)$ If $\mathcal{L}(G)\cong P_{n}$, then $G\cong P_{n+1}$.
\end{fact}

\begin{fact}\label{f2-7}\rm(\cite{pest2019}\rm)
Let $G$ be a connected nontrivial graph. Then $\mathcal{L}(G)$ is regular if and only if $G$ is regular or biregular.
\end{fact}

\begin{fact}\label{f2-8}\rm(\cite{ligz2022}\rm)
Let $G$ be a connected nontrivial graph with maximum degree $\Delta$, minimum degree $\delta$.
If $e=uv\in E_{G}$, then $e\in V_{\mathcal{L}(G)}$, $d_{\mathcal{L}(G)}(e)=d_{G}(u)+d_{G}(v)-2$
and $\max\{2\delta-2,1\}\leq \delta_{\mathcal{L}(G)}\leq \Delta_{\mathcal{L}(G)}\leq 2\Delta-2$,
with left equality if and only if $G$ is $\max\{2\delta-2,1\}$-regular, with right equality if and only if $G$ is $2\Delta-2$-regular.
\end{fact}

\section{Sharp bounds for the SDD index of graphs}

Vasilyev \cite{vasi2014} obtained some bounds for the SDD index of graphs, including the following lower bound of Theorem \ref{t3-1}.
\begin{theorem}\label{t3-1}
Let $G\in \mathcal{G}_{n,m}$. Then $2m\leq SDD(G)\leq m(n-1+\frac{1}{n-1})$,
with left equality if and only if $G$ is regular, right equality if and only if $G\cong S_{n}$.
\end{theorem}
\begin{proof}
By Lemma \ref{l2-2}, one has
$$SDD(G)=\sum\limits_{uv\in E_{G}}(\frac{d_{u}}{d_{v}}+\frac{d_{v}}{d_{u}})\geq 2m,$$
with equality if and only if $G$ is regular.
$$ SDD(G)=\sum\limits_{uv\in E_{G}}(\frac{d_{u}}{d_{v}}+\frac{d_{v}}{d_{u}})\leq \sum\limits_{uv\in E_{G}}(n-1+\frac{1}{n-1})=m(n-1+\frac{1}{n-1}),$$
with equality if and only if $G\cong S_{n}$.
\end{proof}

Since the numbers of cycle $\eta=m-n+1\geq 0$, thus $m\geq n-1$. By Theorem \ref{t3-1}, we have
\begin{corollary}\label{c3-2}
Let $G\in \mathcal{G}_{n}$. Then $SDD(G)\geq 2(n-1)$, with equality if and only if $G\cong K_{2}$.
\end{corollary}

$ID(G)=\sum\limits_{u\in V_{G}}\frac{1}{d_{u}}$ is called the inverse degree index \cite{fajt1987}.
\begin{theorem}\label{t3-3}
Let $G\in \mathcal{G}_{n}$ with maximum degree $\Delta$ and minimum degree $\delta$. Then
$$ \delta^{2}\cdot ID(G)\leq SDD(G)\leq \Delta^{2}\cdot ID(G),$$
with both equalities if and only if $G$ is regular.
\end{theorem}
\begin{proof}
By the definition of SDD index
\begin{eqnarray*}
SDD(G) & = & \sum_{uv\in E_{G}}(\frac{d_{u}}{d_{v}}+\frac{d_{v}}{d_{u}}) \\
& = & \sum_{uv\in E_{G}}(\frac{1}{d_{v}^{2}}+\frac{1}{d_{u}^{2}})d_{u}d_{v}  \\
& \geq & \delta^{2}\sum_{uv\in E_{G}}(\frac{1}{d_{v}^{2}}+\frac{1}{d_{u}^{2}}) \\
& = & \delta^{2}\cdot ID(G),
\end{eqnarray*}
with equality if and only if $G$ is regular.

The proof of the upper bound is similar, we omit it.
\end{proof}

\begin{theorem}\label{t3-4}
Let $G$ be a graph with $|E_{G}|=m$, maximum degree $\Delta$ and minimum degree $\delta$. Then
$$ SDD(G)\leq m(\frac{\Delta}{\delta}+\frac{\delta}{\Delta}),$$
with equality if and only if $G$ is regular or biregular.
\end{theorem}
\begin{proof}
Suppose that $1\leq \delta \leq d_{v}\leq d_{u}\leq \Delta$, and by Lemma \ref{l2-2}, we have
\begin{eqnarray*}
SDD(G) & = & \sum_{uv\in E_{G}}(\frac{d_{u}}{d_{v}}+\frac{d_{v}}{d_{u}}) \\
& \leq & \sum_{uv\in E_{G}}(\frac{\Delta}{\delta}+\frac{\delta}{\Delta})  \\
& = & m(\frac{\Delta}{\delta}+\frac{\delta}{\Delta}),
\end{eqnarray*}
with equality if and only if $d_{u}=\Delta$ and $d_{v}=\delta$ for all $uv\in E_{G}$, i.e., $G$ is regular or biregular.
\end{proof}

\begin{corollary}\label{c3-5}
Let $G\in \mathcal{G}_{n,m}$ with maximum degree $\Delta\leq n-2$. Then
$$ SDD(G)< m(n-2+\frac{1}{n-2}).$$
\end{corollary}
\begin{proof}
Suppose that $1\leq \delta \leq \Delta \leq n-2$, by Theorem \ref{t3-4}, we have
$ SDD(G)\leq m(n-2+\frac{1}{n-2})$ with equality if and only if $G$ is $(n-2,1)$-biregular,
which is a contradiction with $G$ is a connected graph.
Thus $ SDD(G)< m(n-2+\frac{1}{n-2})$.
\end{proof}

\begin{theorem}\label{t3-6}
Let $G$ be a graph with $|E_{G}|=m$, maximum degree $\Delta$ and minimum degree $\delta$. Then
$$ SDD(G)\geq \frac{2\delta^{2}m^{\frac{\alpha+1}{\alpha}}}{(M_{2}^{\alpha}(G))^{\frac{1}{\alpha}}},\ \
SDD(G)\geq \frac{\delta^{2}(2m)^{\frac{\alpha+1}{\alpha}}}{\Delta(M_{1}^{\alpha}(G))^{\frac{1}{\alpha}}}
$$
with both equalities if and only if $G$ is regular.
\end{theorem}
\begin{proof}
By the definition of SDD index, we have
\begin{eqnarray*}
\frac{1}{m}SDD(G) & = & \frac{1}{m}\sum_{uv\in E_{G}}(\frac{d_{u}^{2}+d_{v}^{2}}{d_{u}d_{v}}) \\
& \geq & \left(\prod_{uv\in E_{G}}\frac{d_{u}^{2}+d_{v}^{2}}{d_{u}d_{v}} \right)^{\frac{1}{m}}\\
& \geq & \left( 2^{m}\delta^{2m} \prod_{uv\in E_{G}}\frac{1}{d_{u}d_{v}} \right)^{\frac{1}{m}},
\end{eqnarray*}
with first equality if and only if $\frac{d_{u}^{2}+d_{v}^{2}}{d_{u}d_{v}}$ is a constant for any $uv\in E_{G}$, second equality if and only if $d_{u}=d_{v}=\delta$ for all $uv\in E_{G}$.
Thus
\begin{eqnarray*}
(SDD(G))^{\alpha} & \geq & (2m)^{\alpha}\delta^{2\alpha} \left(\prod_{uv\in E_{G}}(\frac{1}{d_{u}d_{v}})^{\alpha} \right)^{\frac{1}{m}}\\
& \geq & (2m)^{\alpha}\delta^{2\alpha}\cdot \frac{m}{\sum\limits_{uv\in E_{G}}(d_{u}d_{v})^{\alpha}}\\
& = &  \frac{2^{\alpha}m^{\alpha+1}\delta^{2\alpha}}{M_{2}^{\alpha}(G)},
\end{eqnarray*}
with first equality if and only if $G$ is regular, second equality if and only if $d_{u}d_{v}$ is a constant for any $uv\in E_{G}$.
Thus $SDD(G)\geq\frac{2\delta^{2}m^{\frac{\alpha+1}{\alpha}}}{(M_{2}^{\alpha}(G))^{\frac{1}{\alpha}}}$
with equality if and only if $G$ is regular.

By Lemma \ref{l2-3}, $M_{2}^{\alpha}(G)\leq  \frac{\Delta^{\alpha}}{2}M_{1}^{\alpha+1}(G)$ with equality if and only if $G$ is regular.
Thus
$SDD(G)\geq \frac{\delta^{2}(2m)^{\frac{\alpha+1}{\alpha}}}{\Delta(M_{1}^{\alpha}(G))^{\frac{1}{\alpha}}}
$
with equality if and only if $G$ is regular.
\end{proof}

$F(G)=\sum\limits_{uv\in E_{G}}(d_{u}^{2}+d_{v}^{2})$ is called the forgotten index \cite{fugu2015}.
\begin{theorem}\label{t3-7}
Let $G$ be a graph with $|E_{G}|=m$. Then
$$ SDD(G)\geq \frac{2m^{2}}{M_{2}(G)},\ \
SDD(G)\geq \frac{4m^{2}}{F(G)}
$$
with both equalities if and only if $G\cong K_{2}$.
\end{theorem}
\begin{proof}
Since
\begin{eqnarray*}
m & = & \sum_{uv\in E_{G}} \left( \frac{d_{u}d_{v}}{d_{u}^{2}+d_{v}^{2}} \right)^{\frac{1}{2}}
\left( \frac{d_{u}^{2}+d_{v}^{2}}{d_{u}d_{v}} \right)^{\frac{1}{2}} \\
& \leq & \left( \sum_{uv\in E_{G}}\frac{d_{u}d_{v}}{d_{u}^{2}+d_{v}^{2}} \right)^{\frac{1}{2}}
\left( \sum_{uv\in E_{G}}\frac{d_{u}^{2}+d_{v}^{2}}{d_{u}d_{v}} \right)^{\frac{1}{2}},
\end{eqnarray*}
with equality if and only if $\frac{d_{u}^{2}+d_{v}^{2}}{d_{u}d_{v}}$ is a constant for any $uv\in E_{G}$.

Since $d_{u}\geq 1$ for any $u\in V_{G}$, then
$\sum\limits_{uv\in E_{G}}\frac{d_{u}d_{v}}{d_{u}^{2}+d_{v}^{2}}\leq \frac{1}{2}\sum\limits_{uv\in E_{G}}d_{u}d_{v}=\frac{1}{2}M_{2}(G)$, with equality if and only if $d_{u}=d_{v}=1$ for any $uv\in E_{G}$.
Thus $SDD(G)\geq \frac{2m^{2}}{M_{2}(G)}$ with equality if and only if $G\cong K_{2}$.

Since $d_{u}\geq 1$ for any $u\in V_{G}$, then
$\frac{d_{u}d_{v}}{d_{u}^{2}+d_{v}^{2}}\leq \frac{d_{u}^{2}+d_{v}^{2}}{4}$, with equality if and only if $d_{u}=d_{v}=1$ for any $uv\in E_{G}$.
Then $\sum\limits_{uv\in E_{G}}\frac{d_{u}d_{v}}{d_{u}^{2}+d_{v}^{2}}\leq \frac{1}{4}\sum\limits_{uv\in E_{G}}d_{u}^{2}+d_{v}^{2}=\frac{1}{4}F(G)$.
Thus $SDD(G)\geq \frac{4m^{2}}{F(G)}$ with equality if and only if $G\cong K_{2}$.
\end{proof}

In the following, we consider the connected graphs with minimal SDD index.

\begin{theorem}\label{t3-8}
Let $G\in \mathcal{G}_{n,m}$. Then

$(i)$ $SDD(G)\geq 2$, with equality if and only if $G\cong K_{2}$;

$(ii)$ There is no such graphs with $2<SDD(G)\leq 4$;

$(iii)$ If $4<SDD(G)\leq 6$, then $G\in \{S_{3},C_{3}\}$ with $SDD(S_{3})=5$ and $SDD(C_{3})=6$;

$(iv)$ If $6<SDD(G)\leq 8$, then $G\in \{P_{4},C_{4}\}$ with $SDD(P_{4})=7$ and $SDD(C_{4})=8$.
\end{theorem}
\begin{proof}
$(i)$ By Theorem \ref{t3-1}, $SDD(G)\geq 2m\geq 2$, with equality if and only if $G\cong K_{2}$.

Suppose that $n\geq 3$, then we have
$(ii)$ If $2<SDD(G)\leq 4$, then $4\leq 2(n-1)\leq 2m\leq SDD(G)\leq 4$, then $n=3$ and $m=2$.
Thus $G\cong S_{3}$, while $SDD(S_{3})=5>4$, which is a contradiction.

$(iii)$ If $4<SDD(G)\leq 6$, then $4\leq 2(n-1)\leq 2m\leq SDD(G)\leq 6$, then $n=3$ or $4$ and $m\leq 3$. Thus $G\in \{S_{3}, S_{4},P_{4},C_{3}\}$, while $SDD(S_{3})=5$, $SDD(S_{4})=10>6$, $SDD(P_{4})=7>6$, $SDD(C_{3})=6$. Thus $G\in \{S_{3},C_{3}\}$.

$(iv)$ If $6<SDD(G)\leq 8$, then $4\leq 2(n-1)\leq 2m\leq SDD(G)\leq 8$, then $n=3$ or $4$ or $5$ and $m\leq 4$. If $n=3$ and $m\leq 4$, then $G\in \{S_{3},C_{3}\}$ which is a contradiction with $SDD(G)\leq 8$. If $n=4$ and $m\leq 4$, then $G\in \{S_{4},C_{4},P_{4},C_{3}^{*}\}$, where $C_{3}^{*}$ is the graph obtained from $C_{3}$ by adding a pendent vertex to one vertex of $C_{3}$. $SDD(S_{4})=10>8$, $SDD(P_{4})=7$, $SDD(C_{4})=8$, $SDD(C_{3}^{*})=9+\frac{2}{3}>8$. Thus $G\in \{P_{4},C_{4}\}$ in this case.

If $n=5$ and $m\leq 4$, since $m\geq n-1=4$, thus $m=4$.
then $G\in \{P_{5},P_{4}^{*},S_{5}\}$, where $P_{4}^{*}$ is the graph obtained from $P_{4}$ by adding a pendent vertex to one vertex with degree two of $P_{4}$. $SDD(P_{5})=9>8$, $SDD(P_{4}^{*})=11+\frac{1}{3}>8$, $SDD(S_{5})=17>8$. Thus $G=\emptyset$ in this case.
\end{proof}

The inverse problem for the SDD index is also interesting, thus we propose the following problem.
\begin{problem}\label{p3-1}
Solve the inverse problem for the SDD index of graphs or chemical graphs.
\end{problem}

We call $u_{0}v_{0}\in E_{G}$ is a minimal edge in $G$ if $d_{u_{0}}\leq d_{u}$ for all $u\in N_{G}(u_{0})\setminus \{v_{0}\}$ and $d_{v_{0}}\leq d_{u}$ for all $u\in N_{G}(v_{0})\setminus \{u_{0}\}$.

\begin{theorem}\label{t3-9}
Let $G$ be a graph with a minimal edge $u_{0}v_{0}$. Let $G^{*}=G-u_{0}v_{0}$. Then
$$ SDD(G^{*})>SDD(G)-\frac{(d_{u_{0}})^{2}+(d_{v_{0}})^{2}}{d_{u_{0}}d_{v_{0}}}.$$
\end{theorem}
\begin{proof}
Since $G^{*}=G-u_{0}v_{0}$, then $V_{G}=V_{G^{*}}$.
Let $d_{u}\in V_{G}$ and $d_{u}^{*}\in V_{G^{*}}$, then
$d_{u_{0}}^{*}=d_{u_{0}}-1$, $d_{v_{0}}^{*}=d_{v_{0}}-1$ and $d_{u}^{*}=d_{u}$ for all $u\in V_{G}\setminus \{u_{0},v_{0}\}$.

Let $E_{G}\supseteq E_{0}=\{uv\in E_{G}|u\notin\{u_{0},v_{0}\},v\notin\{u_{0},v_{0}\}\}$. Then
\begin{eqnarray*}
& & SDD(G)-SDD(G^{*}) \\
& = & \sum_{uv\in E_{0}} \left(\frac{d_{u}^{2}+d_{v}^{2}}{d_{u}d_{v}}-\frac{(d_{u}^{*})^{2}
+(d_{v}^{*})^{2}}{d_{u}^{*}d_{v}^{*}}\right)+ \sum_{u\in N_{G^{*}}(u_{0})} \left(\frac{d_{u_{0}}^{2}+d_{u}^{2}}{d_{u_{0}}d_{u}}-\frac{(d_{u_{0}}^{*})^{2}
+(d_{u}^{*})^{2}}{d_{u_{0}}^{*}d_{u}^{*}}\right)\\
& \quad &  +\sum_{u\in N_{G^{*}}(v_{0})} \left(\frac{d_{v_{0}}^{2}+d_{u}^{2}}{d_{v_{0}}d_{u}}-\frac{(d_{v_{0}}^{*})^{2}
+(d_{u}^{*})^{2}}{d_{v_{0}}^{*}d_{u}^{*}}\right)+ \frac{d_{u_{0}}^{2}+d_{v_{0}}^{2}}{d_{u_{0}}d_{v_{0}}}\\
& = & \sum_{u\in N_{G^{*}}(u_{0})} \left(\frac{d_{u_{0}}^{2}+d_{u}^{2}}{d_{u_{0}}d_{u}}-\frac{(d_{u_{0}}^{*})^{2}
+(d_{u}^{*})^{2}}{d_{u_{0}}^{*}d_{u}^{*}} \right)+\sum_{u\in N_{G^{*}}(v_{0})} \left( \frac{d_{v_{0}}^{2}+d_{u}^{2}}{d_{v_{0}}d_{u}}-\frac{(d_{v_{0}}^{*})^{2}
+(d_{u}^{*})^{2}}{d_{v_{0}}^{*}d_{u}^{*}} \right)\\
& \quad & + \frac{d_{u_{0}}^{2}+d_{v_{0}}^{2}}{d_{u_{0}}d_{v_{0}}}.
\end{eqnarray*}

Since $u_{0}v_{0}$ is a minimal edge in $G$, then $1\leq d_{u_{0}}\leq d_{u}$ for all $u\in N_{G}(u_{0})\setminus \{v_{0}\}$.
Since $d_{u}(d_{u_{0}}-1)(d_{u_{0}}^{2}+d_{u}^{2})-d_{u}d_{u_{0}}((d_{u_{0}}-1)^{2}+d_{v}^{2})
=d_{u}(d_{u_{0}}^{2}-d_{u}^{2}-1)<0$, then $\frac{d_{u_{0}}^{2}+d_{u}^{2}}{d_{u_{0}}d_{u}}-\frac{(d_{u_{0}}^{*})^{2}
+(d_{u}^{*})^{2}}{d_{u_{0}}^{*}d_{u}^{*}}<0$. Similarly, we also have $\frac{d_{v_{0}}^{2}+d_{u}^{2}}{d_{v_{0}}d_{u}}-\frac{(d_{v_{0}}^{*})^{2}
+(d_{u}^{*})^{2}}{d_{v_{0}}^{*}d_{u}^{*}}<0$.
Thus we have $SDD(G^{*})>SDD(G)-\frac{(d_{u_{0}})^{2}+(d_{v_{0}})^{2}}{d_{u_{0}}d_{v_{0}}}$.
\end{proof}

\section{Sharp bounds for the SDD index of line graphs}

It is obvious that $SDD(\mathcal{L}(G))=0$ if and only if $G$ is a trivial graph, i.e., $G\cong K_{2}$. Thus in the following, we suppose $G\ncong K_{2}$.

\begin{theorem}\label{t4-1}
Let $G$ be a graph with $|E_{G}|=m$. Then

$(i)$ If $G\ncong P_{m+1}$, then $SDD(\mathcal{L}(G))\geq 2m$, with equality if and only if $G\in \{S_{4},C_{m}\}$;

$(ii)$ If $G\ncong K_{2}$, then $SDD(\mathcal{L}(G))\leq (\frac{1}{2}M_{1}(G)-m_{G})(m_{G}-1+\frac{1}{m_{G}-1})$, with equality if and only if $G\in \{P_{3},P_{4}\}$.
\end{theorem}
\begin{proof}
By Lemma \ref{l2-4}, $m_{\mathcal{L}(G)}\geq m_{G}=n_{\mathcal{L}(G)}$ with equality if and only if $\mathcal{L}(G)$ is a unicyclic graph. By Theorem \ref{t3-1}, $SDD(\mathcal{L}(G))\geq 2m_{\mathcal{L}(G)}\geq 2m$, with equality if and only if $\mathcal{L}(G)$ is regular unicyclic graph, i.e., $\mathcal{L}(G)\cong C_{m}$, then $G\in \{S_{4},C_{m}\}$.

By Fact \ref{f2-6}, Theorem \ref{t3-1} and Lemma \ref{l2-5}, we have that if $G\ncong K_{2}$, then $SDD(\mathcal{L}(G))\leq (\frac{1}{2}M_{1}(G)-m_{G})(m_{G}-1+\frac{1}{m_{G}-1})$, with equality if and only if $G\in \{P_{3},P_{4}\}$.
\end{proof}

Combine Theorem \ref{t3-3} and Fact \ref{f2-6}, we have
\begin{theorem}\label{t4-2}
Let $G$ be a graph with maximum degree $\Delta$ and minimum degree $\delta$. If $G\ncong K_{2}$ , then
$$\max\{4(\delta-1)^{2},1\}\cdot ID(\mathcal{L}(G))\leq SDD(\mathcal{L}(G))\leq 4(\Delta-1)^{2}\cdot ID(\mathcal{L}(G)),$$
with left equality if and only if $G$ is regular or $G\cong S_{3}$, right equality if and only if $G$ is regular.
\end{theorem}

\begin{theorem}\label{t4-3}
Let $G$ be a graph with $|E_{G}|=m$, maximum degree $\Delta$ and minimum degree $\delta$. If $G\ncong K_{2}$ , then
$$M_{1}(G)-2m\leq SDD(\mathcal{L}(G))\leq \frac{1}{2}(M_{1}(G)-2m)\left( \frac{2\Delta-2}{\max\{ 2\delta-2,1\}}+ \frac{\max\{ 2\delta-2,1\}}{2\Delta-2}  \right),$$
with left equality if and only if $G$ is regular or biregular, right equality if and only if $G\cong P_{4}$ or $G$ is regular.
\end{theorem}
\begin{proof}
By Lemma \ref{l2-5} and Theorem \ref{t3-1}, we have $SDD(\mathcal{L}(G))\geq 2m_{\mathcal{L}(G)}=M_{1}(G)-2m$, with equality if and only if $\mathcal{L}(G)$ is regular, i.e.,
$G$ is regular or biregular.

By Theorem \ref{t3-4},Lemma \ref{l2-2}, Lemma \ref{l2-5} and Fact \ref{f2-8}, we have
\begin{eqnarray*}
SDD(\mathcal{L}(G)) & \leq & m_{\mathcal{L}(G)} \left( \frac{\Delta_{\mathcal{L}(G)}}{\delta_{\mathcal{L}(G)}}+\frac{\delta_{\mathcal{L}(G)}}
{\Delta_{\mathcal{L}(G)}}  \right) \\
& = &  \frac{1}{2}(M_{1}(G)-2m) \left( \frac{\Delta_{\mathcal{L}(G)}}{\delta_{\mathcal{L}(G)}}+\frac{\delta_{\mathcal{L}(G)}}
{\Delta_{\mathcal{L}(G)}}  \right)   \\
& \leq & \frac{1}{2}(M_{1}(G)-2m)\left( \frac{2\Delta-2}{\max\{ 2\delta-2,1\}}+ \frac{\max\{ 2\delta-2,1\}}{2\Delta-2}  \right),
\end{eqnarray*}
with first equality if and only if $\mathcal{L}(G)$ is regular or biregular, second equality if and only if $\delta_{\mathcal{L}(G)}=\max\{ 2\delta-2,1\}$ and $\Delta_{\mathcal{L}(G)}=2\Delta-2$.

If $G\cong P_{4}$ or $G$ is regular, we have the equality holds.
$SDD(\mathcal{L}(P_{4}))=SDD(S_{3})=5=\frac{1}{2}(10-2\times 3)(\frac{2\times2-2}{1}+\frac{1}{2\times2-2})$, and
$SDD(\mathcal{L}(G))=2m_{\mathcal{L}(G)}=M_{1}(G)-2m=\frac{1}{2}(M_{1}(G)-2m)\left( \frac{2\Delta-2}{\max\{ 2\delta-2,1\}}+ \frac{\max\{ 2\delta-2,1\}}{2\Delta-2}  \right)$.

In the following, we suppose that $\mathcal{L}(G)$ is regular or biregular, and $\delta_{\mathcal{L}(G)}=\max\{ 2\delta-2,1\}$ and $\Delta_{\mathcal{L}(G)}=2\Delta-2$.

{\bf Case 1}. $\delta=1$.

Then $\delta_{\mathcal{L}(G)}=\max\{ 2\delta-2,1\}=1$.
Since $G\ncong K_{2}$, then $\Delta\geq 2$ and $\Delta_{\mathcal{L}(G)}=2\Delta-2\geq 2$.
Then $\mathcal{L}(G)$ is a $(\Delta_{\mathcal{L}(G)},1)$-biregular graph.
Thus $\mathcal{L}(G)\cong S_{\Delta_{\mathcal{L}(G)}+1}$ with $n_{\mathcal{L}(G)}=\Delta_{\mathcal{L}(G)}+1\geq 3$.
By Fact \ref{f2-6}, we have $G\cong P_{4}$.

{\bf Case 2}. $\delta\geq 2$.

In this case, $\mathcal{L}(G)$ is regular or biregular, and $2\delta-2= \delta_{\mathcal{L}(G)}\leq \Delta_{\mathcal{L}(G)}=2\Delta-2$.
If $\mathcal{L}(G)$ is biregular, we have $\delta_{\mathcal{L}(G)}< \Delta_{\mathcal{L}(G)}$,
then $\delta<\Delta$, which is a contradiction with the definition of biregular graphs and
$2\delta-2= \delta_{\mathcal{L}(G)}\leq \Delta_{\mathcal{L}(G)}=2\Delta-2$.
Then $\mathcal{L}(G)$ is regular, thus $2\delta-2= \delta_{\mathcal{L}(G)}= \Delta_{\mathcal{L}(G)}=2\Delta-2$. Thus $G$ is regular in this case.
\end{proof}

In the following, we consider the Nordhaus-Gaddum-type results for the SDD index of a graph $G$ and its line graph $\mathcal{L}(G)$.

\begin{corollary}\label{c4-4}
Let $G$ be a graph with maximum degree $\Delta$ and minimum degree $\delta$. If $G\ncong K_{2}$ , then
$$M_{1}(G)\leq SDD(G)+SDD(\mathcal{L}(G))\leq \frac{1}{2}M_{1}(G)\left( \frac{2\Delta-2}{\max\{ 2\delta-2,1\}}+ \frac{\max\{ 2\delta-2,1\}}{2\Delta-2}  \right),$$
with both equalities if and only if $G$ is regular.
\end{corollary}
\begin{proof}
Combine Theorem \ref{t3-1} and Theorem \ref{t4-3}, we have $SDD(G)+SDD(\mathcal{L}(G))\geq M_{1}(G)$, with equality if and only if $G$ is regular.

For the upper bound of $SDD(G)+SDD(\mathcal{L}(G))$, we consider the following two cases.

{\bf Case 1}. $\delta=1$.

Then $G$ is not a regular graph $(G\ncong K_{2})$, thus $\Delta\geq 2$.
By Theorem \ref{t3-4} and Theorem \ref{t4-3}, we have
\begin{eqnarray*}
& & SDD(G)+SDD(\mathcal{L}(G)) \\
& < & m \left( \frac{\Delta^{2}+1}{\Delta}  \right)+ \frac{1}{2}(M_{1}(G)-2m)\left( \frac{4(\Delta-1)^{2}+1}{2(\Delta-1)} \right)\\
& = &  \frac{1}{4}M_{1}(G)\left( \frac{4(\Delta-1)^{2}+1}{\Delta-1} \right)+m \left( \frac{\Delta^{2}+1}{\Delta}-\frac{4(\Delta-1)^{2}+1}{2(\Delta-1)} \right) \\
& \leq & \frac{1}{4}M_{1}(G)\left( \frac{4(\Delta-1)^{2}+1}{\Delta-1} \right).
\end{eqnarray*}

{\bf Case 2}. $\delta\geq 2$.

It is easy to proof that $\frac{\Delta^{2}+\delta^{2}}{\Delta\delta}\leq \frac{(\Delta-1)^{2}+(\delta-1)^{2}}{(\Delta-1)(\delta-1)}$. Then
\begin{eqnarray*}
& & SDD(G)+SDD(\mathcal{L}(G)) \\
& \leq & m \left( \frac{\Delta^{2}+\delta^{2}}{\Delta\delta}  \right)+ \frac{1}{2}(M_{1}(G)-2m)\left( \frac{(\Delta-1)^{2}+(\delta-1)^{2}}{(\Delta-1)(\delta-1)} \right)\\
& \leq &  m \left( \frac{(\Delta-1)^{2}+(\delta-1)^{2}}{(\Delta-1)(\delta-1)} \right)+ \frac{1}{2}(M_{1}(G)-2m)\left( \frac{(\Delta-1)^{2}+(\delta-1)^{2}}{(\Delta-1)(\delta-1)} \right) \\
& = & \frac{1}{2}M_{1}(G)\left( \frac{(\Delta-1)^{2}+(\delta-1)^{2}}{(\Delta-1)(\delta-1)} \right),
\end{eqnarray*}
with equality if and only if $G$ is regular.

Thus we have
$SDD(G)+SDD(\mathcal{L}(G))\leq \frac{1}{2}M_{1}(G)\left( \frac{2\Delta-2}{\max\{ 2\delta-2,1\}}+ \frac{\max\{ 2\delta-2,1\}}{2\Delta-2}  \right)$, with equality if and only if $G$ is regular.
\end{proof}

\begin{theorem}\label{t4-5}
Let $G$ be a graph with $|E_{G}|=m$, maximum degree $\Delta$ and minimum degree $\delta$. Then
$$SDD(\mathcal{L}(G))\geq \frac{\Delta \delta^{2}\chi_{\alpha+1}(G)}{(\Delta-1)^{2}(\chi_{\alpha}(G))^{\frac{1}{\alpha}}},$$
with equality if and only if $G$ is regular.
\end{theorem}
\begin{proof}
It is obvious that the conclusion holds for $\delta=1$. In the following, we consider $\delta\geq 2$.

By Fact \ref{f2-8}, Lemma \ref{l2-5} and Theorem \ref{t3-6}, we have
$$  SDD(\mathcal{L}(G))\geq \frac{(\delta_{\mathcal{L}(G)})^{2}(2m_{\mathcal{L}(G)})^{\frac{\alpha+1}{\alpha}}}
{\Delta_{\mathcal{L}(G)}(M_{1}^{\alpha}(\mathcal{L}(G)))^{\frac{1}{\alpha}}}
\geq \frac{2(\delta-1)^{2}(M_{1}(G)-2m)^{\frac{\alpha+1}{\alpha}}}
{(\Delta-1)(M_{1}^{\alpha}(\mathcal{L}(G)))^{\frac{1}{\alpha}}},  $$
with equality if and only if $\mathcal{L}(G)$ is $(2\Delta-2)$-regular.

Since $M_{1}^{\alpha}(\mathcal{L}(G))\leq (\frac{\Delta-1}{\Delta})^{\alpha}\chi_{\alpha}(G)$ for $\alpha>0$, with equality if and only if $G$ is $\Delta$-regular \cite{ligz2022}.
Thus
$SDD(\mathcal{L}(G))\geq \frac{\Delta \delta^{2}\chi_{\alpha+1}(G)}{(\Delta-1)^{2}(\chi_{\alpha}(G))^{\frac{1}{\alpha}}}$,
with equality if and only if $G$ is regular.
\end{proof}

\begin{theorem}\label{t4-6}
Let $G$ be a graph with $|E_{G}|=m$ and maximum degree $\Delta$. If $G\ncong K_{2}$, then
$$SDD(\mathcal{L}(G))>\frac{\Delta^{3}(M_{1}(G)-2m)^{2}}{(\Delta-1)^{3}\chi_{3}(G)}.$$
\end{theorem}
\begin{proof}
Since $\frac{d_{u}+d_{v}-2}{d_{u}+d_{v}}\leq \frac{\Delta-1}{\Delta}$ for any vertices $u,v\in V_{G}$ with maximum $\Delta$, the equality holds if and only if $d_{u}=d_{v}=\Delta$.
Then $d_{\mathcal{L}(G)}(uv)=d_{u}+d_{v}-2\leq (d_{u}+d_{v})\frac{\Delta-1}{\Delta}$.

Since
$F(\mathcal{L}(G))=\sum\limits_{uv\in V_{\mathcal{L}(G)}}(d_{\mathcal{L}(G)}(uv))^{3}
\leq (\frac{\Delta-1}{\Delta})^{3}\sum\limits_{uv\in E_{G}}(d_{u}+d_{v})^{3}=\frac{\Delta^{3}(M_{1}(G)-2m)^{2}}{(\Delta-1)^{3}\chi_{3}(G)}$,
with equality if and only if $G$ is $\Delta$-regular.

Then By Lemma \ref{l2-5} and Theorem \ref{t3-7}, we have
$$ SDD(\mathcal{L}(G))\geq \frac{4(m_{\mathcal{L}(G)})^{2}}{F(\mathcal{L}(G))}
= \frac{(M_{1}(G)-2m)^{2}}{F(\mathcal{L}(G))}\geq \frac{\Delta^{3}(M_{1}(G)-2m)^{2}}{(\Delta-1)^{3}\chi_{3}(G)}, $$
with first equality if and only if $\mathcal{L}(G)\cong P_{2}$, i.e., $G\cong P_{3}$, second equality if and only if $G$ is regular. This is a contradiction.
$SDD(\mathcal{L}(G))>\frac{\Delta^{3}(M_{1}(G)-2m)^{2}}{(\Delta-1)^{3}\chi_{3}(G)}$.
\end{proof}

In the following, we consider the line graphs with minimal SDD index.
Combine the Fact \ref{f2-6} and Theorem \ref{t3-8}, we have the following result.
The proof of \ref{t4-7} is similar to Theorem \ref{t3-8}, we omit it.
\begin{theorem}\label{t4-7}
Let $G$ be a graph with $G\ncong K_{2}$. Then

$(i)$ $SDD(\mathcal{L}(G))\geq 2$, with equality if and only if $G\cong P_{3}$;

$(ii)$ There is no such graphs with $2<SDD(\mathcal{L}(G))\leq 4$;

$(iii)$ If $4<SDD(\mathcal{L}(G))\leq 6$, then $G\in \{P_{4},C_{3},S_{4}\}$ with $SDD(\mathcal{L}(P_{4}))=5$ and $SDD(\mathcal{L}(C_{3}))=SDD(\mathcal{L}(S_{4}))=6$;

$(iv)$ If $6<SDD(\mathcal{L}(G))\leq 8$, then $G\in \{P_{5},C_{4}\}$ with $SDD(\mathcal{L}(P_{5}))=7$ and $SDD(\mathcal{L}(C_{4}))=8$.
\end{theorem}

The inverse problem for the SDD index of line graphs $\mathcal{L}(G)$ is also interesting, thus we propose the following problem.
\begin{problem}\label{p4-1}
Solve the inverse problem for the SDD index of line graphs $\mathcal{L}(G)$.
\end{problem}

\begin{theorem}\label{t4-8}
Let $G\in \mathcal{G}_{n}$ with maximum degree $\Delta$ and minimum degree $\delta$, and $G\ncong K_{2}$. Then
$$\frac{SDD(\mathcal{L}(G))}{SDD(G)}\leq \frac{\Delta^{2}-\delta}{4\delta} \left( \frac{4(\Delta-1)^{2}+(\max\{ 2\delta-2,1\})^{2}}{(\Delta-1)\cdot \max\{ 2\delta-2,1\}} \right),$$
with equality if and only if $G$ is regular.
\end{theorem}
\begin{proof}
We know that $G\ncong K_{2}$, $\Delta_{\mathcal{L}(G)}\leq 2\Delta-2$, $\delta_{\mathcal{L}(G)}\geq \max\{ 2\delta-2,1\}$.

Since $M_{1}(G)\leq \frac{2m\Delta^{2}}{\delta}$, with equality if and only if $G$ is regular \cite{ligz2022}.
By Theorem \ref{t3-4} and Lemma \ref{l2-2}, we have
\begin{eqnarray*}
SDD(\mathcal{L}(G)) & \leq & m_{\mathcal{L}(G)} \left( \frac{\Delta_{\mathcal{L}(G)}}{\delta_{\mathcal{L}(G)}}+\frac{\delta_{\mathcal{L}(G)}}
{\Delta_{\mathcal{L}(G)}}  \right) \\
& = &  \frac{1}{2}(M_{1}(G)-2m) \left( \frac{\Delta_{\mathcal{L}(G)}}{\delta_{\mathcal{L}(G)}}+\frac{\delta_{\mathcal{L}(G)}}
{\Delta_{\mathcal{L}(G)}}  \right)   \\
& \leq & \frac{1}{2}(M_{1}(G)-2m)\left( \frac{2\Delta-2}{\max\{ 2\delta-2,1\}}+ \frac{\max\{ 2\delta-2,1\}}{2\Delta-2}  \right)   \\
& \leq &  (\frac{m\Delta^{2}}{\delta}-m)\left( \frac{2\Delta-2}{\max\{ 2\delta-2,1\}}+ \frac{\max\{ 2\delta-2,1\}}{2\Delta-2}  \right)         \\
& = &  2m(\frac{\Delta^{2}-\delta}{2\delta})\left( \frac{2\Delta-2}{\max\{ 2\delta-2,1\}}+ \frac{\max\{ 2\delta-2,1\}}{2\Delta-2}  \right)         \\
& \leq &  SDD(G)\cdot (\frac{\Delta^{2}-\delta}{2\delta})\left( \frac{2\Delta-2}{\max\{ 2\delta-2,1\}}+ \frac{\max\{ 2\delta-2,1\}}{2\Delta-2}  \right) ,
\end{eqnarray*}
with equality if and only if $G$ is regular.
\end{proof}

\section{Conclusions}
In a recent paper \cite{fudg2018}, Furtula et al. determined the quality of SDD index exceeds that of some more popular VDB indices, in particular that of the GA index. They shown a close connection between the SDD index and the earlier well-established GA index. Thus it is meaningful and important to consider the chemical and mathematical properties of the SDD index.

Liu et al. \cite{lpli2020} determined the minimum and second minimum SDD index of tricyclic graphs. By the way, using a similar way of \cite{deeb2018,lhua2022}, we can also determine minimum and second minimum SDD index of tetracyclic (chemical) graphs.

\end{document}